\documentclass[titlepage,draft,11pt]{article}
\usepackage[usenames,dvipsnames]{xcolor}
\usepackage{pstricks,pst-plot,pst-node} \usepackage{amssymb,amsthm} 
\usepackage[reqno, namelimits, sumlimits]{amsmath}
\usepackage[a4paper]{geometry}

\newtheorem{theorem}{Theorem}[section]

\newtheorem{proposition}[theorem]{Proposition}
\newtheorem{corollary}[theorem]{Corollary}

\theoremstyle{definition}

\theoremstyle{remark}
\newtheorem{remark}[theorem]{Remark}

\def\R{{\mathbb R}}

\def\cA{{\mathcal A}}
\def\cE{{\mathcal E}}

\def\cH{{\mathcal H}}
\def\cM{{\mathcal M}}

\def\cL{{\mathcal L}}

\def\TcT{({\mathbf T} (t))_{t\geq 0}}

\def\TcT0{({\mathbf T}_0 (t))_{t\geq 0}}

\def\L1{L^1 (\R_+ )}

\numberwithin{equation}{section}

\title{A Liapunov function approach to  the stabilization of  second order coupled systems}

\author{Alain Haraux\vspace{1ex}\\ 
{\normalsize Sorbonne Universit\'es, UPMC Univ Paris 06, CNRS, UMR 7598} \\
{\normalsize Laboratoire Jacques-Louis Lions}\\{\normalsize  4, place Jussieu 75005, Paris, France.}\\
{\normalsize e-mail: \texttt{haraux@ann.jussieu.fr}}
\and Mohamed Ali Jendoubi\vspace{1ex}\\ 
{\normalsize Universit\'e de Carthage}\\
{\normalsize Institut Pr\'eparatoire aux Etudes Scientifiques et Techniques}\\
{\normalsize B.P. 51,  2070 La Marsa, Tunisia} \\
{\normalsize e-mail: \texttt {ma.jendoubi@fsb.rnu.tn}}}

\begin{document}
\maketitle
\begin{abstract}

\small{In 2002, Fatiha Alabau, Piermarco Cannarsa and Vilmos Komornik investigated the extent of asymptotic stability of the null solution for weakly coupled partially damped equations of the second order in time. The main point is that the damping operator acts only on the first component and, whenever it is bounded, the coupling is not strong enough to produce an exponential decay in the energy space associated to the conservative part of the system.  As a consequence, for initial data in the energy space, the rate of decay is not exponential. Due to the nature of the result it seems at first sight impossible to obtain the asymptotic stability result by the classical Liapunov method.  Surprisingly enough, this turns out to be possible and we exhibit, under some compatibility conditions on the operators, an explicit class of Liapunov functions which allows to do 3 different things: \\

1) When the problem is reduced to a stable finite dimensional space, we recover the exponential decay by a single differential inequality and we estimate the logarithmic decrement of the solutions with worst (slowest) decay. The estimate is optimal at least for some values of the parameters. \\

2) We explain the form of the stability result obtained by the previous authors when the coupling operator is a multiple of the identity, so that the decay is not exponential.\\

3) We obtain new exponential decay results when the coupling operator is strong enough (in particular unbounded). The estimate is again sharp for some solutions.}

\vspace{6ex}
\noindent{\bf Mathematics Subject Classification 2010 (MSC2010):}
35B40, 49J15, 49J20.\vspace{6ex}


\end{abstract}

\section{Introduction}

\bigskip 

In 2002, Fatiha Alabau, Piermarco Cannarsa and Vilmos Komornik  published the paper \cite{ACK02} in which they investigated the extent of asymptotic stability  of the null solution for $t>0$ for weakly coupled partially damped equations of the type  $$ u'' + A_1u + Bu' + Cv = v'' + A_2v + Cu = 0$$ where $A_1, A_2, B   \,\,\hbox{and} \,\, C$ are positive self-adjoint operators  satisfying additional conditions. The main point is that the damping operator acts only on the first component $u$ and when $A_1, A_2$ are comparable coercive unbounded operators while $B, C$ are coercive and bounded, the coupling is not strong enough to produce an exponential decay in the energy space associated to the conservative part of the system.  As a consequence, for initial data in the energy space, decay takes place in a weaker function space and the rate of decay is not exponential. Moreover, due to the nature of the result it seems impossible to obtain the asymptotic stability result by the classical Liapunov method which we now recall in a few lines.  In  \cite{Liap} (1892), Liapunov defined and investigated the dynamical stability of equilibrium solutions to differential systems of the form
 $$U'(t) = F(U(t))$$ where $F \in C^2 (\R^N)$. Given $a\in F^{-1}(0)$ he proved that $a$ is asymptotically stable (in fact exponentially stable) as soon as all the eigenvalues of the square matrix 
 $ M= DF(a)$  have {\it negative } real parts. This result is now classical and has been recalled in quite a few books, with different sorts of proofs depending on the applications the authors had in mind as well as their cultural background.  The original proof of Liapunov consisted in considering first the linearized equation $$ Y' = MY(t) \eqno(LIN)$$ for which $0$ is an exponentially stable equilibrium. Under the hypothesis on the eigenvalues, it is not difficult to see that all solutions of $(LIN)$ tend to $0$ as $t$ tends to infinity.  Then by considering a basis of $\R^N$ , it follows easily that for some $T>0$ we have $||\exp(TM)|| < 1$. Then by a classical division argument we find 
$$ \forall t\ge 0, \quad ||\exp(tM)|| \le  C e^{-\delta t}$$ for some $C\ge 1$ and $\delta>0$, thereby proving exponential stability of $0$ for the linearized equation. Is seems that at the time of Liapunov (and even much later) it was not natural to use the potential well argument for the nonlinear perturbation equation by using Duhamel's variation of constants formula. Therefore Liapunov looked for a renorming
allowing to get the same estimate with $C=1$, in which case a direct potential well argument in differential form becomes possible. The following quadratic function
$$ \Phi(z) = \int_0^\infty |\exp(sM)z|^2 ds $$ provides  a solution of the problem. Indeed for any solution $ Y(t) = \exp(tM) Y_0 $ of $(LIN)$ we have 
$$ \frac{d}{dt} \Phi(Y(t)) = \frac{d}{dt} \int_0^\infty |\exp(sM)\exp(tM) Y_0|^2 ds $$ 
$$ = \frac{d}{dt} \int_0^\infty |\exp(s+t)M)Y_0|^2 ds = \frac{d}{dt} \int_t^\infty |\exp(\tau)M)Y_0|^2 d\tau = 
- |Y(t)| ^2$$ By the equivalence of norms on the finite dimensional space $\R^N$ we see immediately that the new norm defined by $ ||z|| = \Phi(z)^{1/2}$ is a solution. Therefore in finite dimensions it is always possible to prove exponential stability by means of a renorming in which the norm $||Y(t)|| $ satisfies a differential inequality of the form
$$  \frac{d}{dt} ||Y(t)|| ^2 \le - \gamma ||Y(t)|| ^2$$  
In particular this seems to be a practical way of estimating the logarithmic decrement (or characteristic numbers by Liapunov's terminology) of solutions. However even the case of the simplest system \begin{equation} \label{ODE}u'' + \lambda u + bu' + cv = v'' + \lambda v + cu = 0\end{equation}  were $\lambda >0$, $b>0$ and $c\not=0$ can have any sign shows the difficulty of the problem. Standard manipulation gives the identity 
$$ \frac{d}{dt} (\lambda u^2 + \lambda v^2 + 2cuv + u'^2 + v'^2 ) = -2bu'^2 \le 0$$  Assuming $c^2 <\lambda ^2 $, the function 
$$ F(u,v,w,z) = \lambda u^2 + \lambda v^2 + 2cuv  + w^2 + z^2$$ is a positive definite quadratic form. Since 
$F(u,v,u',v') $ is non-increasing along the trajectories, the 4 components $(u,v,u',v')$ are bounded and we are in a good position to apply the invariance principle (cf. e.g.\cite{SD, HJ}). Indeed let $(u,v)$ be a solution for which $F(u,v,u',v') $ is constant. Then $2bu'^2 = 0$ implies $u'= 0$, hence $u$ is constant and $u''= 0$. Then by the first equation $v = -\frac{\lambda}{c}u$ is also constant. Finally since by the hypothesis  $c^2 <\lambda ^2 $ , the stationary system $ \lambda u  + cv =  \lambda v + cu = 0 $ has no non-trivial solution, we conclude that  $u= v= 0$  and therefore $(0,0,0,0)$  is asymptotically stable, implying exponential stability as recalled above. Now an interesting problem occurs:  the quadratic form $\Phi$ introduced by Liapunov cannot be computed since we do not have access to an explicit formula for the semi-group (the characteristic equation has degree 4!) We know, however, that the form can be computed on a basis of $\frac{4\times (4+1)}{2} = 10 $ monomials in $(u,v,w,z)$. The challenge is therefore to find one of the strict Liapunov functions (they form a non-empty open set in the space of coefficients) by a direct method, hoping that it will enlighten the nature of stability also in the more complicated (for instance infinite dimensional) cases. The object of the present paper is to carry out this specific program. More precisely, in section 2, we exhibit a class of strict Liapounov functions for the above scalar ODE. In sections 3 and 4, we  evaluate by two different approaches  the ``worst" characteristic number of solutions. In section 5, we generalize the construction to a class of strongly coupled second order equations with a linear damping acting on only one of the two components and section 6 is devoted to examples. Finally, in the last section 7, we recover one of the main results from \cite{ACK02} by using a weakened notion of strict Liapunov functions. This method seems to be applicable to more general situations and the approach can be used to obtain explicit estimates, at the expense of complicated but not impossible refinements of our calculations.

\section{A Liapunov function for the scalar case}
 In this section we consider the (real) scalar coupled system

\begin{equation} \label{ODE}
\left\{ \begin{array}{ll}
u''+u'+\lambda u +cv=0 & \\[2mm]
v''+\lambda v +cu=0 &  
\end{array} \right. \end{equation}
where $\lambda$ and $c$ are such that $0< \vert c\vert<\lambda$. The damping coefficient is set to $1$ for simplicity but a time scale change reduces general damping terms $bu'$ to this case. In order to shorten the formulas, let us introduce for each solution $(u,v)$ of \eqref{ODE}, its total energy
$$ \cE(u,u',v,v') = \frac12\left[   u'^2+   v'^2+\lambda(u^2+v^2)\right]+ c uv $$
Then we have for all $t\ge0$ $$ \frac{d}{dt} \cE(u,u',v,v') = -u'^2 $$ Our first main result is the foliowing

 \begin{proposition} \label{Liapscal}{For any $p>1$ fixed and for all $\varepsilon>0$ small enough the quadratic form 
  \begin{equation}\label{FtLiapunovSystDiff}
H_\varepsilon=\cE-\varepsilon vv'+ p\varepsilon uu'+\frac{(p+1)\lambda\varepsilon}{2c}(u'v-uv')
\end{equation} 
is a strict Liapunov function for \eqref{ODE}.}
\end{proposition}
 \begin{proof}  First of all we note that the derivative of the skew product involves $-v^2$. Indeed $$ \frac{d}{dt} (u'v-uv') = (u''v-uv'')= -v (u'+cv+\lambda u)+ u(cu + \lambda v) $$
$$ = c(u^2-v^2)-u'v $$ Then we find easily 
\begin{eqnarray*} && \frac{d}{dt} [-vv'+ p uu'+\frac{(p+1)\lambda}{2c}(u'v-uv')]\\
&=& pu'^2-v'^2 + v(cu+\lambda v) - pu (u'+cv+\lambda u) \\
 & &\quad + \frac{(p+1)\lambda}{2}(u^2-v^2)- 
\frac{(p+1)\lambda}{2c}u'v \\
&=& pu'^2-v'^2  -u'(pu + \frac{(p+1)\lambda}{2c}v)- \frac{(p-1)}{2} [\lambda(u^2+v^2) +2cuv].
\end{eqnarray*} 
The end of the proof is now nearly obvious. First we have 
$$ \lambda(u^2+v^2) +2cuv \ge (\lambda- |c|) (u^2+v^2) $$ Moreover we have for some constant $K>0$ 

$$ |u'(pu + \frac{(p+1)\lambda}{2c}v)| \le  \frac{(p-1)}{4}(\lambda- |c|) (u^2+v^2)  + Ku'^2$$  so that 
$$\frac{d}{dt} [-vv'+ p uu'+\frac{(p+1)\lambda}{2c}(u'v-uv')] \le (p+K) u'^2-v'^2$$ $$ -\frac{(p-1)}{4}(\lambda- |c|) (u^2+v^2) $$ The conclusion follows immediately.\end{proof}

 \begin{remark}{\rm The only missing term in this quadratic form is $u'v'$. This was predictable since its derivative does not seem to contain any interesting term. Moreover it is usual that the Liapunov function is a small perturbation of the energy. The term in $uu'$ seems to be mandatory since it is what we need in the uncoupled case to produce the emergence of a $-u^2$ term . The term in $ - vv' $  is added to produce  a $-v'^2$ by differentiation. It is then remarkable that a multiple of the wronskian-like  skew product $u'v-uv'$ is sufficient to produce the emergence of a $-v^2$ term and at the same time compensate the ``junk terms" coming from the other differentiated terms.}\end{remark}

   \section{On the logarithmic decrement as a function of the coefficients}
  
  The Liapunov function constructed in the previous section provides a theoretical tool to evaluate the logarithmic decrement of the semi-group generated by the scalar system  \eqref{ODE}, which can be defined as the upper bound of the set of  $\gamma>0$ for which  $ \exp({\gamma t})  ||T(t)||$  is bounded for $t\ge 0$, or equivalently as the logarithmic decrement (resp. characteristic number in the sense of Liapunov) of the most slowly decaying solutions. However, due to the large number of inequalities which we need to combine to exploit this Liapunov function, it seems difficult to get a sharp estimate of the decrement in all cases. 
  \bigskip
  
  \noindent In order to have a more precise idea of the dependance of $\gamma$ on the coefficients 
  $c,  \lambda $ it is useful to look for qualitative information based on the characteristic polynomial, even though the roots are in general impossible to compute. The characteristic polynomial P associated to \eqref{ODE} is easily computed:
  $$ P(\zeta) = (\zeta^2+\lambda)(\zeta^2+\zeta+\lambda)- c^2$$ 
  Several  remarks are in order 
   \begin{remark}{\rm  The logarithmic decrement never exceeds $\frac{1}{4}$. Indeed us denote by $ \zeta_j$ the 4 characteristic numbers (eventually counted with their multiplicity) of \eqref{ODE} and let us set $\rho_j : = - Re(\zeta_j).$ Since $ \sum _{j= 1}^4 \rho_j  = 1$ we have $ \inf_{j}\rho_i  \le \frac{1}{4}$.}\end{remark}
\begin{remark}{\rm For $|c|$ close enough to $\lambda$, the characteristic equation has some real roots. Indeed the function $$ G(\theta) = \sqrt {(\lambda + \theta^2)(\lambda -\theta + \theta^2)}$$ is continuous decreasing for small positive values of $\theta$ , and for any $c$ close enough to $\lambda$, the number $ \zeta=  - G^{-1}(c)$ is a real (negative) eigenvalue of the generator. Here we recover the fact that as $|c|$ approaches $\lambda$, the stabilization effect disappears and the logarithmic decrement tends to $0$. }\end{remark}
 \begin{remark}{\rm A number $\zeta = s+ ia$ with $s<0,\  a\in\R$ is a solution of the characteristic equation of and only if  it satisfies the two equations \begin{equation}\label{imaginary} a[ 4s^3 + 4(\lambda - a^2)s + 3 s^2 + \lambda - a^2] = 0\end{equation}  and 
  \begin{equation}\label{real} s^4 - 4a^2s^2 + 2(\lambda-a^2)s^2 + (\lambda-a^2)^2 + s(s^2+ (\lambda-a^2) -2a^2s = c^2\end{equation} If $a\not=0$, the equation \eqref{imaginary} reduces to 
\begin{equation}\label{a} a^2 = \lambda + \frac{4s^3+3s^2}{1+4s} \end{equation} }
\end{remark}
The next proposition completes remark 3.1.
\begin{proposition} The logarithmic decrement is always strictly less than $\frac{1}{4}$. On the other hand for any $\varepsilon>0$ small enough, there exists $\lambda >0$ and $c\in (0, \lambda)$ such that  the logarithmic decrement is equal to $\frac{1}{4}-\varepsilon$. 
\end{proposition}
 \begin{proof} For the proof of the first assertion we reason by contradiction. Assuming that the decrement is equal to  $\frac{1}{4}$ means that   $ \inf_{j}\rho_i = \frac{1}{4}$. In particular $\rho_j \ge \frac{1}{4}$ for all and since $ \sum _{j= 1}^4 \rho_j  = 1$  this yields 
 $$ \forall j\in (1,2,3,4), \quad \rho_j = \frac{1}{4}. $$  Hence all roots are of the form $ \zeta_j = -\frac{1}{4} + ia_j$  However if  $ \zeta = -\frac{1}{4} + ia$ is a root of $P$, we must have  $a= 0$ . Indeed if $a\not= 0$, \eqref{imaginary} implies $4s^3 + (\lambda - a^2)(4s+1) + 3 s^2 = 0$ and since $4s+1= 0$ we deduce $4s^3 + 3 s^2 = 0$, contradicting $s = - \frac{1}{4}$. This means that  $ \zeta_j = -\frac{1}{4} $ for all $ j\in (1,2,3,4)$, hence $P(\zeta) = (\zeta^2+\lambda)(\zeta^2+\zeta+\lambda)- c^2= (\zeta+\frac{1}{4})^4. $ Identification of the coefficients provide an immediate contradiction, thereby proving the claim. \bigskip 
 
 For the proof of the second assertion a more technical argument is needed. First we  look for $\lambda$ and $c\in (0, \lambda)$ such that the equation $ P(\zeta) = 0$ has a solution of the form $$ -\frac{1}{4}+\varepsilon + ia, \quad  a\in \R,\ a\not = 0.$$ In this case the conjugate number $ -\frac{1}{4}+\varepsilon - ia $ is also a root, and the sum of the two remaining root equals  $ -\frac{1}{2}-2\varepsilon $, their product is also known . If these roots appear to be not real, their common real part will equal $-\frac{1}{4}-\varepsilon$ and  $ \frac{1}{4}-\varepsilon$ will be exactly equal to the logarithmic decrement. We conclude the proof in two steps. \bigskip
 
 {\bf Step 1.} We look for $\lambda$ and $c\in (0, \lambda)$. Since we want   $a\not = 0$, we have the formula
 \eqref{a} and by substituting the value of $a^2$ in \eqref{real} we obtain the remaining necessary and sufficient condition on $c$ in the form 
 $$ c^2 = s^4 -(4\lambda + 6 \frac{4s^3+3s^2}{1+4s} )s^2 +( \frac{4s^3+3s^2}{1+4s})^2 + \frac{s^2(s- 3s^2)}{1+4s}  -2(  \lambda + \frac{4s^3+3s^2}{1+4s} )^2 s $$ with $s = -\frac{1}{4}+\varepsilon $ .  A precise inspection of the terms shows that for $\varepsilon$ small , $$ c^2 = (\frac{1}{32\varepsilon})^2 + \frac{\lambda}{2} + O(\frac{1}{\varepsilon})$$ Now we can make (for instance) the choice $\lambda = \frac{1}{16\varepsilon}$, so that asymptotically, $ c\sim \frac{\lambda}{2}$. The only thing remaining to prove is that the remaining roots are not real. 
 \bigskip 
 
 {\bf Step 2.} The remaining roots are not real for $\varepsilon$ small. Indeed, these roots are the solutions of the equation 
 $$X^2 -SX+ P = 0$$  with 
 $$ S= -\frac{1}{2}-2\varepsilon, \quad P =  \frac{\lambda^2-c^2}{(\frac{1}{4}-\varepsilon)^2 + a^2 }$$ We claim that for $\varepsilon$ small enough the discriminant $S^2-4P$ is negative. Since $S$ is bounded, it is sufficient to prove that $P$ tends to $+\infty$. Now we have $ \lambda^2-c^2\sim \frac{3}{4}\lambda^2 \sim \frac{3}{1024 \varepsilon^2}$ and $(\frac{1}{4}-\varepsilon)^2 + a^2 = \lambda + O(\frac{1}{\varepsilon}) = O(\frac{1}{\varepsilon}) $. The conclusion follows immediately.

 \end{proof}

\section {Optimality  in some range of parameters}

It is interesting (and perhaps a bit surprising ) to note that the method of proof of Proposition \ref{Liapscal} gives a result very close to optimality in some range of parameters, specifically when the largest possible logarithmic decrement is almost achieved. More precisely we have, assuming for definiteness $c>0$
\begin{proposition} \label{1/4}{As $\frac{c}{\lambda}$ tends to $0$ and $\frac{c}{\lambda^{1/2}}$ tend to infinity, the logarithmic decrement (as evaluated by the method of proof of Proposition \ref{Liapscal} ) tends to the highest possible value $\frac{1}{4}$.}
\end{proposition}
 \begin{proof}  We introduce 
 $$ F: = F(u,u',v,v') = \frac12\left[   u'^2+   v'^2+\lambda(u^2+v^2)\right] $$ Following the notation of Section 1, it is easy to check, assuming $p\ge 1$, that 
 $$ [ \frac{\lambda -c}{2\lambda}  - \varepsilon (\frac{p}{2{\lambda}^{1/2}}+\frac{(p+1){\lambda}^{1/2}}{4c} )]F \le H  \le [ \frac{\lambda +c}{2\lambda}  + \varepsilon (\frac{p}{2{\lambda}^{1/2}}+\frac{(p+1){\lambda}^{1/2}}{4c} )]F$$  On the other hand, starting from the formula  $$ H_\varepsilon' = -u'^2 +  \varepsilon \frac{d}{dt} [-vv'+ p uu'+\frac{(p+1)\lambda}{2c}(u'v-uv')] $$ we find  $$ H_\varepsilon' = -(1- p\varepsilon) u'^2-\varepsilon v'^2  -\varepsilon u'(pu + \frac{(p+1)\lambda}{2c}v)- \varepsilon\frac{(p-1)}{2} [\lambda(u^2+v^2) +2cuv] $$ hence 
 
 $$ H_\varepsilon' \le  -(1- p\varepsilon) u'^2-\varepsilon v'^2  -\varepsilon u'(pu + \frac{(p+1)\lambda}{2c}v)- \varepsilon\frac{(p-1)}{2} [(\lambda- c) (u^2+v^2)]  $$ In order to appraise the third term of the RHS, we introduce a constant $\gamma \in (0, 1)$ which will be later taken arbitrarily small and we write $$ |u'(pu + \frac{(p+1)\lambda}{2c}v)| \le \frac{\gamma (p-1)}{2} (\lambda- c) u^2 + \frac {{p}^2}{2\gamma (p-1)(\lambda- c)} u'^2$$ $$ +  \frac{\gamma (p-1)}{2} (\lambda- c) v^2 + \frac {({\frac{(p+1)\lambda}{2c}})^2}{2\gamma (p-1)(\lambda- c)} u'^2$$ 
 $$ \le \frac{\gamma (p-1)}{2} (\lambda- c) (u^2 +v^2) + \frac {4{p}^2c^2 + (p+1)^2 \lambda^2}{8\gamma (p-1)(\lambda- c) c^2} u'^2$$ so that we find 
  $$ H_\varepsilon' \le  - [1- \varepsilon (p+\frac {4{p}^2c^2 + (p+1)^2 \lambda^2}{8\gamma (p-1)(\lambda- c) c^2}) ] u'^2-\varepsilon v'^2  - (1-\gamma) \varepsilon\frac{(p-1)}{2} [(\lambda- c) (u^2+v^2)]  $$ In order to make the extreme right term equal to $- \varepsilon \lambda (u^2+v^2)$ we determine $p$ by the equation
  $$(1-\gamma) (p-1)(\lambda- c) = 2\lambda $$ hence 
  $$ p= 1+ \frac{2}{(1-\gamma)(1-\theta)};\quad \theta: = \frac{c}{\lambda}$$ We observe that as $\gamma$ and $\theta$ tend to $0$, $p$ will tend to $3$.  Our first goal being to achieve the inequality $H_\varepsilon'\le -\varepsilon F$, we now require 
  $$1- \varepsilon (p+\frac {4{p}^2c^2 + (p+1)^2 \lambda^2}{8\gamma (p-1)(\lambda- c) c^2}) = \varepsilon $$ hence 
  $$ \varepsilon = \frac{1}{1+ p+\frac {4{p}^2c^2 + (p+1)^2 \lambda^2}{8\gamma (p-1)(\lambda- c) c^2}}  =  \frac{1}{1+ p+\frac {1-\gamma}{16\gamma\lambda }[4{p}^2 +\frac{(p+1)^2}{\theta^2}]} $$ We observe that under the choice $\gamma : \frac{\lambda^{1/2}}{c}$ which tends to $0$ by hypothesis, $p$ stabilizes to $3$ and $\gamma \lambda $ becomes infinite. Moreover $\gamma \lambda \theta^2 = \gamma \frac {c^2}{\lambda}= \frac{c}{\lambda^{1/2}}$ also tends to infinity. Therefore the limiting value of $$ \varepsilon  =  \frac{1}{1+ p+\frac {1-\gamma}{16\gamma\lambda }[4{p}^2 +\frac{(p+1)^2}{\theta^2}]} $$ is $\frac{1}{1+3} = \frac{1}{4} $. Moreover from the inequality $H_\varepsilon'\le -\varepsilon F$, it follows that 
  $H_\varepsilon' \le -\delta H_\varepsilon$ with 
  $$ \delta = \frac{\varepsilon}{ \frac{\lambda +c}{2\lambda}  + \varepsilon (\frac{p}{2{\lambda}^{1/2}}+\frac{(p+1){\lambda}^{1/2}}{4c})} $$ which reduces asymptotically to $$\delta \sim 2\varepsilon$$ so that the limiting value of $\delta$ is $1/2$.  It is not difficult, in the range that we considered, to see that $F$ is bounded by a constant times $H$ . Thus 
  
  $$ F(t)\le C_1H_\varepsilon(t)\le C_2 e^{-\delta t} $$ since $F(t)$ measures the square of the norm of the solution, the limiting value of the logarithmic decrement of solutions is $1/4$ as claimed.
 
 \end{proof}

\section {The strongly coupled case}

In this section we generalize the scalar system in a framework which concerns finite dimensional and infinite dimensional systems as well.  Let $ A $   be a closed, self-adjoint, positive coercive operator on a separable Hilbert space $H$. with domain $D(A)$.  We denote by $(u,v)$ the inner product of two vectors $u,v$ in $H$ and by $\vert u\vert$ the $H$ norm of $u$. Let $V=D(A^{\frac12})$ endowed with the norm given by
$$\forall u\in V,\quad \Vert u\Vert=\vert A^{\frac12} u\vert .$$

 The topological dual of $H$ is identified with $H$, therefore
$$V\subset H=H'\subset V'$$ with continuous and dense imbeddings. 
 Let $C \in L(V, V') $   satisfy the following conditions 
 \begin{equation}\Vert C\Vert_{L(V, V')} < 1\end{equation}  
 We consider the second order evolution system
 \begin{equation}\label{SystOndeFortAbstrait}
\left\{ \begin{array}{ll}
u'' +u'+A u +C v=0, & \\[2mm]
 v''+A v+C ^*u=0. &  
\end{array} \right. 
\end{equation}
which can be rewritten as the first order system

\begin{equation}\label{SystAbstrait}
\left\{ \begin{array}{ll}
u'-w=0 & \\[2mm]
v'-z=0 & \\[2mm]
w'+Au+w+Cv=0 & \\[2mm]
z'+Av+C^*u=0. &  
\end{array} \right. 
\end{equation} 
We introduce $U:=(u,v,w,z)\in V\times V\times H\times H=:{\cH}$ and the space $\cH$ is endowed with the inner product $\langle \cdot,\cdot\rangle_{\cH}$ defined by
$$\langle (u,v,w,z),(\widehat{u},\widehat{v},\widehat{w},\widehat{z})\rangle _{\cH}=\langle Au,\widehat{u}\rangle+\langle Av,\widehat{v}\rangle+\langle w,\widehat{w}\rangle+\langle z,\widehat{z}\rangle+$$
$$+\langle Cu,\widehat{u}\rangle+\langle C\widehat{v},u\rangle.$$
 We define an unbounded operator $\cA$ on $\cH$ by the formulas
$$D(\cA)=\{(u,v,w,z)\in V^4,\ Au+Cv\in H,\ Av+C^* u\in H\}$$ 
and
$$\cA(u,v,w,z)=(-w, -z,Au+Cv+w,Av+C^*u),\quad \forall (u,v,w,z)\in D(\cA),$$ so that \eqref{SystAbstrait} is formally equivalent to
$$ U' + \cA U(t) = 0.$$
One has
\begin{eqnarray*}\langle \cA U,U\rangle_{\cH} &=& -\langle Aw,u\rangle-\langle Az,v\rangle+\langle Au+w+Cv,w\rangle+\\
&& \quad +\langle Av+C^*u,z\rangle-\langle Cz,u\rangle-\langle Cv,w\rangle\\
&=&\vert w\vert^2
\end{eqnarray*} 
Hence $\cA\geq 0$ on $D(\cA)$. Actually $\cA$ is maximal monotone. Indeed, to prove this, according to the general theory, cf e.g.\cite{{Lum-Ph}, Minty, OpMon, H0} and the references therein,  it suffices to prove that $\cA + I$ is onto. The system
$$\cA U+U=F=(f,g,\varphi,\psi)\in\cH $$
reduces to
 \begin{equation}\label{SystavantLM}
\left\{ \begin{array}{ll}
w=u-f & \\[2mm]
z=v-g & \\[2mm]
Au+Cv+2u=\varphi+2f \quad(\in H)& \\[2mm]
Av+C^*u+v=\psi+g. &  
\end{array} \right. 
\end{equation} 
We introduce the form
$$\Phi(u,v)=\frac12 (\vert A^{\frac12} u\vert^2+\vert A^{\frac12} v\vert^2)+\langle Cu,v\rangle.$$
The two last equations of \eqref{SystavantLM} reduce to
\begin{eqnarray*}D\Phi(u,v)+(2u,v)&=&(\varphi+2f,\psi+g)\\
&=&D\Phi_1(u,v)
\end{eqnarray*} 
where $\Phi_1(u,v)=\Phi(u,v)+\vert u\vert^2+\frac12\vert v\vert^2$, $D\Phi$ denotes the derivative of $\Phi$, $D\Phi\in\cL(V,V')$.\\
 $D\Phi$ is a symmetric operator as well as $D\Phi_1$, since $D\Phi$ is coercive, so is $D\Phi_1$. By Lax-Milgram theorem $$D\Phi_1(V)=V',$$ in particular $H\subset D\Phi_1(V)$ and this solves \eqref{SystavantLM}. Moreover, since $\varphi+2f$ and $\psi+g\in H$, we find $Au+Cv\in H, Av+C^*u\in H$, so that $U=(u,v,w,z)\in D(\cA)$. In particular, as a consequence of the general theory of semi-groups we find   \begin{proposition} For any $U_0 = (u_0,v_0,w_0,z_0)\in \cH$ there exists a unique solution $U \in C(\R^+, \cH) \cap C^1(\R^+, H\times H\times V'\times V')$ of with $U(0) = U_0$. Moreover , introducing
$$H_0(u,v,,w,z)=\frac12 (\Vert u\Vert^2+\Vert v\Vert^2+\vert w\vert^2+\vert z \vert^2)+\langle Cv, u \rangle.$$
  Then all solutions of the system \eqref{SystOndeFortAbstrait} are bounded and we have
  $$H_0(u(t),v,(t),w(t),z(t)=H_0(u(t),v,(t),u'(t),v'(t) \in C^1(\R^+) $$ with
 $$\frac{d}{dt}H_0(u(t),v(t),u'(t),w'(t)=-\vert u'\vert^2$$  Moreover if 
 $U_0 = (u_0,v_0,w_0,z_0)\in D(\cA)$, then $U \in C^1(\R^+, \cH)$, in particular $u, v$ are 
 in $ C^1(\R^+, V) \cap C^2(\R^+, H).$
\end{proposition}
 The main result of this section is the following
 \begin{theorem} \label {strong coup}  Assume that $C$ satisfies the following additional conditions :\begin{equation}\label{Hyp1-C^-1}\ker C = 0,\quad  H \subset C(V) \,\,\hbox{and} \,\,C^{-1}\in L(H, V)\end{equation} 
  \begin{equation}\label{Hyp2-C^-1} V' \subset C(H) \,\,\hbox{and} \,\,C^{-1}\in L(V', H)\end{equation} 
 \begin{equation}\label{Hyp3-C^-1} AC^{-1}- C^{-1}A \in L(H, H)\end{equation}  
 
\medskip\noindent meaning that the operator $ D = AC^{-1}- C^{-1}A \in L(V, V')$ is in fact bounded
 for the 
$H$-norm with values in $H$ and can therefore be extended on the whole of H as a bounded operator. 
Then for any $p>1$ fixed and  for all $\varepsilon>0$ small enough the quadratic form $H_{\varepsilon} = H_{\varepsilon}(u,v,w,z)$ defined by
\begin{equation}\label{Liap}H_{\varepsilon}=H_0-\varepsilon (v,z)+p\varepsilon(u,w)+\frac{(p+1)\varepsilon}{2}[\langle AC^{-1}w,v\rangle-\langle C^{-1}Au,z\rangle]\end{equation}
is a strict Liapunov functional. In particular the semi-group generated by 
\eqref{SystOndeFortAbstrait} is exponentially damped in $V\times V \times H \times H$. \end{theorem}\begin{proof}  We start with the case of strong solutions with $U_0 = (u_0,v_0,w_0,z_0)\in D(\cA)$. In this case we have $$\frac{d}{dt} [-(v, v')+ p (u, u')] = p|u'|^2-|v|'^2 + (v, A v +C^*u) - p(u, u'+Au +Cv) $$
$$ = p|u'|^2-|v|'^2 - p(u, u') + ||v||^2 - p||u||^2 -(p-1)(Cv, u)$$ 
On the other hand for strong solutions, the functions $\langle AC^{-1}u'(t),v(t)\rangle $ and \\
$\langle C^{-1}Au,v'(t)\rangle $ belong to $C^1(\R^+)$ with $$ \frac{d}{dt} [\langle AC^{-1}u',v\rangle-\langle C^{-1}Au,v'\rangle] = \langle (AC^{-1}- C^{-1}A) u', v'\rangle + \langle AC^{-1}u'',v\rangle-\langle C^{-1}Au,v''\rangle]$$
$$ = \langle Du', v'\rangle + \langle -AC^{-1}u' -AC^{-1}Au - Av,v\rangle + \langle C^{-1}Au, Av + C^*u\rangle]$$ 

$$ = \langle Du', v'\rangle - \langle -AC^{-1}u' ,v\rangle - ||v^2|| + ||u||^2 $$ 

Then we find easily 
$$\frac{d}{dt} \{-(v,v')+p(u,u')+\frac{(p+1)}{2}[\langle AC^{-1}u',v\rangle-\langle C^{-1}Au,v'\rangle] \}= p|u'|^2-|v|'^2  -  p(u, u')  $$ $$ + \frac{(p+1)}{2} (\langle Du', v'\rangle -  \langle -AC^{-1}u' ,v\rangle ) - \frac{(p-1)}{2} [||u||^2 + ||v^2|| + (Cv, u)]. $$
The end of the proof is now rather staightforward. Since $D\in L(H)$ and $ AC^{-1} \in L(H, V')$   by using the Cauchy-Schwarz inequality in all terms involving $u'$ we can achieve, as in the ODE case, a choice of $\varepsilon$ independent of the initial data so that 
$$\frac{d}{dt} H_{\varepsilon}
\le - \frac{\varepsilon}{2} (|u'|^2+|v|'^2) - \frac{(p-1)}{4}{\varepsilon} (1- ||C|| _{V, V'} ) (||u||^2+||v||^2) $$ 
Moreover since the  RHS of the last equality is continuous for the topology of ${\cH}$, by interating on a small time interval and passing to the limit by density, it is easy to see that the modified energy $H_{\varepsilon}$  is in fact in $C^1(\R^+)$ even in the case of weak solutions, so that our final inequality is valid in general. The conclusion follows immediately.\end{proof}

   \begin{remark} \label{ConditionsC}{\rm  Let us comment briefly about the meaning of the conditions \eqref {Hyp1-C^-1}, \eqref{Hyp2-C^-1} and \eqref{Hyp3-C^-1}.} The two first conditions express the fact that not only $C$ is regular, but its inverse $ C^{-1} $ has a smoothing effet at least equal to the smoothing effect of $A^{-1/2}.$  The third condition is an extra boundedness condition on the commutator of $A$ and $C$  and is automatically satisfied in the two following cases: \medskip 
   
   1) $C^{-1} \in L(H, D(A))\cap L(D(A)', H)$ 
   
   2) $A $ and $C$ commute with each other. \medskip
   
   \noindent In particular in finite dimensions there is no other condition than invertibility of $C$, and if $A, C$ are two elliptic operators of the same order with the same boundary conditions on a bounded domain, no commutation condition will be required. Finally is may be useful to observe that if $C= cA^{\alpha}$ with $\alpha\ge 0$, the hypotheses will be satisfied if and only if $c\not = 0$ and $\alpha\ge 1/2.$ \end{remark}

\section{ Some examples of strongly coupled systems } This section is devoted to a short list of examples in which Theorem \ref{strong coup} gives exponential decay together with a method to evaluate the logarithmic decrement of the slowest decaying solutions by means of an explicit Liapunov function.  \subsection{Finite dimensional examples} In finite dimensions, there is no condition to add relying the operators $A$ and $C$. A special case is the complex scalar example 

\begin{equation} \label{ODE2}
\left\{ \begin{array}{ll}
u''+u'+\lambda u +icv=0& \\[2mm]
v''+\lambda v -icu =0 &  
\end{array} \right. 
 \end{equation} which can also be written in real form 
 \begin{equation} \label{ODE3}
\left\{ \begin{array}{ll}
u_1''+u_1'+\lambda u_1 - cv_2=0 & \\[2mm]
u_2''+u_2'+\lambda u_2 + cv_1=0 & \\[2mm]
v_1''+\lambda v_1 +cu_2 =0 &  \\[2mm]
v_2''+\lambda v_2 -cu_1 =0 & 
\end{array} \right. 
 \end{equation} and could therefore be treated as the combination of the two real systems
 \begin{equation*} 
\left\{ \begin{array}{ll}
u_1''+u_1'+\lambda u_1 - cv_2=0 & \\[2mm]
v_2''+\lambda v_2 -cu_1 =0 & 
\end{array} \right. \end{equation*} and 
\begin{equation*} 
\left\{ \begin{array}{ll}
u_2''+u_2'+\lambda u_2 + cv_1=0 & \\[2mm]
v_1''+\lambda v_1 +cu_2 =0 & \end{array} \right. 
 \end{equation*}
 
 For the more general system \begin{equation} \label{ODE2}
\left\{ \begin{array}{ll}
u''+u'+\lambda u +(c+id)v=0& \\[2mm]
v''+\lambda v + (c-id)u =0 &  
\end{array} \right. 
 \end{equation}  the Liapunov functions cannot be found so easily by the combination of two scalar systems and the general formula \eqref{Liap} becomes useful. We find 
 
 \begin{equation}
H_\varepsilon=\frac12\left[   u'^2+   v'^2+\lambda(u^2+v^2)\right]+ Re \{\zeta uv+\varepsilon [- vv'+ p uu'+ \frac{(p+1)\lambda}{2\zeta }(u'v-uv')]\}
\end{equation} The choice $p= 3$ leads to the slightly simpler formula  
\begin{equation}
H_\varepsilon=\frac12\left[   u'^2+   v'^2+\lambda(u^2+v^2)\right]+ Re \{\zeta uv+\varepsilon [- vv'+ 3 uu'+ \frac{2\lambda}{\zeta }(u'v-uv')]\}\end{equation} 

\subsection{The wave equation with strong (maximal) coupling} 

Let  $\Omega$ be a bounded open domain of $\R^N$  . Then for any $\gamma \in (0, 1)$, the system

\begin{equation} \label{SystOndeFort}
\left\{ \begin{array}{ll}
\partial_t^2 u-\Delta u+\partial_t u-\gamma\Delta v=0 & \\[2mm]
\partial_t^2 v-\Delta v-\gamma\Delta u=0 &  
\end{array}\right. \end{equation} 
with homogeneous Dirichlet boundary conditions generates an exponentially damped linear semi-group in $V\times V \times H \times H$ with $H= L^2(\Omega)$ 
and $V= H^1_0(\Omega).$ A Liapunov functional is given for $\varepsilon $ small enough by 
$$ 
H_\varepsilon=\frac12\int_{\Omega}\left[   |\partial_t u|^2+  |\partial_t v|^2+|\nabla u|^2+  |\nabla v|^2\right] dx + \gamma \int_{\Omega} \nabla u. \nabla v  dx $$ $$+\varepsilon \int_{\Omega}  ( 3 u\partial_t u -v\partial_t v) dx  + \frac{2\varepsilon} {\gamma}\int_{\Omega}  ( v\partial_t u -u\partial_t v) dx $$

\subsection{The plate equation with structural (minimal) coupling} 
Let  $\Omega$ be a bounded open domain of $\R^N$ with $C^2$ boundary . Then for any $\gamma \in (0, \lambda_1(\Omega))$, the system\begin{equation} \label{SystOndeFort}
\left\{ \begin{array}{ll}
\partial_t^2 u+\Delta^2 u+\partial_t u-\gamma\Delta v=0 & \\[2mm]
\partial_t^2 v +\Delta^2 v-\gamma\Delta u=0 &  
\end{array}\right.
\end{equation} with the  boundary conditions  $u = v= \Delta u = \Delta v = 0$ generates an exponentially damped linear semi-group in $W\times W\times H \times H$ with $H= L^2(\Omega)$ 
and $W= H^2\cap H^1_0(\Omega).$
 A Liapunov functional is given for $\varepsilon $ small enough by 
$$ H_\varepsilon=\frac12\int_{\Omega}\left[   |\partial_t u|^2+  |\partial_t v|^2+|\Delta u|^2+  |\Delta v|^2\right] dx + \gamma \int_{\Omega} \nabla u. \nabla v  dx $$ $$+\varepsilon \int_{\Omega}  ( 3 u\partial_t u -v\partial_t v) dx  + \frac{2\varepsilon} {\gamma}\int_{\Omega}  ( \Delta u\partial_t v -\Delta v\partial_t u) dx $$

\subsection{A string equation with structural (minimal) coupling} 

The system 
\begin{equation} \label{Ondeper}
\left\{ \begin{array}{ll}
\partial_t^2 u-\partial_x^2 u +\partial_t u+ \gamma \partial_x v=0 & \\[2mm]
\partial_t^2 v-\partial_x^2v  -\gamma\partial_x u=0 &  
\end{array}\right. \end{equation} 
on a interval $(0, l)$ generates an exponentially damped linear semi-group in $V\times V \times H \times H$ where  $H$ is the space of $L^2$  functions in $(0, l)$ with mean-value $0$ and $V$ is the space of  $H^1$, l-periodic functions with mean-value $0$ , whenever  $\gamma\not = 0$ with $|\gamma |$ sufficiently small. A Liapunov functional is given for $\varepsilon $ small enough by 
$$ 
H_\varepsilon=\frac12\int_{\Omega}\left[   |\partial_t u|^2+  |\partial_t v|^2+|u|^2+  |v|^2 +|\partial_x u|^2+  |\partial_x v|^2\right] dx + \gamma \int_{\Omega}  u . \partial_x  v  dx $$ $$+\varepsilon \int_{\Omega}  ( 3 u\partial_t u -v\partial_t v) dx  + \frac{2\varepsilon} {\gamma}\int_{\Omega}  ( \partial_x u .\partial_t v - \partial_x v.\partial_t u) dx $$ 

\subsection{The wave equation with strong non-commuting coupling} We give here an example illustrating Remark \ref{ConditionsC}. Let  $\Omega$ be a bounded open domain of $\R^N$ and let $(a, b)$ be two real-valued, measurable, essentially bounded potentials on $\Omega$ with $ \min \{ \min _{x\in \Omega} a(x),  \min _{x\in \Omega} b(x)\} + \lambda_1(\Omega)>0$. Then for any $\gamma >0 $ small enough, the system

\begin{equation} \label{SystOndeFort+ pot}
\left\{ \begin{array}{ll}
\partial_t^2 u -\Delta u + a(x) u+\partial_t u+\gamma(-\Delta v +b(x) v)=0 & \\[2mm]
\partial_t^2 v-\Delta v + a(x) v + \gamma(-\Delta u +b(x)) u=0 &  
\end{array}\right. \end{equation} 
with homogeneous Dirichlet boundary conditions generates an exponentially damped linear semi-group in $V\times V \times H \times H$ with $H= L^2(\Omega)$ 
and $V= H^1_0(\Omega).$ \\

\begin{remark} {\rm  In this theorem, the Laplacian may be replaced by any strongly elliptic self-adjoint operator of order two with smooth coefficients. Here we do not give the formula for the Liapunov functionals since they are a bit more complicated than in the previous examples, but of course the reader can write them easily by applying the general formula \eqref{Liap} with $ A= 
-\Delta  + a(x) I$ and $ C= 
\gamma(-\Delta  + b(x) I)$. The smallness condition on $\gamma$ will depend  on $a, b$. We leave the details to the potentially interested reader.} \end{remark}

\subsection{A plate equation with structural non-commuting coupling} We conclude this Section by a slightly more delicate  example.  Let  $\Omega$ be a bounded open domain of $\R^N$ with $C^2$ boundary and let $m\in L^\infty(\Omega)$ be a non-negative function . Then for any $\gamma \in (0, \lambda_1(\Omega))$, the system\begin{equation} \label{SystOndeFort}
\left\{ \begin{array}{ll}
\partial_t^2 u+\Delta^2 u +m(x)u+\partial_t u-\gamma\Delta v=0 & \\[2mm]
\partial_t^2 v +\Delta^2 v +m(x)v -\gamma\Delta u=0 &  
\end{array}\right.
\end{equation} with the  boundary conditions  $u = v= \Delta u = \Delta v = 0$ generates an exponentially damped linear semi-group in $W\times W\times H \times H$ with $H= L^2(\Omega)$ 
and $W= H^2\cap H^1_0(\Omega).$
Indeed, assuming that $W$ is endowed with the norm given by the formula $$ \Vert u\Vert^2 _W= \int _\Omega (|\Delta u|^2+ m(x)u^2) dx$$ it is easy to check that 
$$ \Vert \Delta u\Vert_{W'}  \le \frac {1}{\lambda_1(\Omega)}\Vert \Delta u\Vert_{H}\le \frac {1}{\lambda_1(\Omega)}\Vert u\Vert _W$$
Moreover, here $C^{-1} = (-\gamma \Delta)^{-1}$ and $A = \Delta^2  +m(x)I $ do not commute, but 
$$ C^{-1}A - AC^{-1}= C^{-1}{\cM}-{\cM}C^{-1} $$ where $\cM$ denotes the operator of multiplication  by $m(x)$ is not only bounded, but even compact as an operator from $H$ to itself. \\

\begin{remark} {\rm  It is even possible to consider the case $ C= \gamma(-\Delta + b(x)I) $ with $b\in C^1(\overline {\Omega})$ } and $b\ge c> - \lambda_1(\Omega) $, but the calculations are more difficult and are left as an exercise. \end{remark}

\section{ The infinite dimensional weakly coupled case } We now consider the system

\begin{equation}\label{SystOndeFaibleAbstrait}
\left\{ \begin{array}{ll}
u'' +u'+A u +c v=0 & \\[2mm]
 v''+ A v+c u=0 &  
\end{array} \right. 
\end{equation}
 and we introduce $$E(u,v,,w,z)=\frac12 (\Vert u\Vert^2+\Vert v\Vert^2+\vert w\vert^2+\vert z \vert^2)+c(u,v).$$
Then all (weak) solutions of the system \eqref{SystOndeFaibleAbstrait} are bounded with
 $$\frac{d}{dt}E(u,v,u',v')=-\vert u'\vert^2$$ 
and we have
 \begin{theorem}   Assume $c\not=0$ and $\vert c\vert < \lambda_1(A)$. Then for any   $p>1$ fixed such that 
 \begin{equation}\label{Hypsurp}\displaystyle \frac{p+1}{p-1}<\frac{\lambda_1}{\vert c\vert}
\end{equation}
and  all $\varepsilon>0$ small enough the quadratic form $H_\varepsilon $ defined by
$$H_\varepsilon(u,v,w,z)=E-\varepsilon\lambda_1 (v,z)_*+p\varepsilon(u,w)+\rho\varepsilon[(w,v)-(u,z)]$$ 
with $ \rho = \frac{(p+1)\lambda_1 }{2c}$  satisfies the inequality  
$$\frac{d}{dt}H_\varepsilon(u,v,u',v')\le - \gamma (p, \varepsilon) \frac12 (\vert u\vert^2+\vert v\vert^2+\Vert u'\Vert_*^2+\Vert v'\Vert_*^2)$$ valid for any weak solution of  \eqref{SystOndeFaibleAbstrait}.
\end{theorem}

 \begin{corollary} For any  solution $(u,v)$ of  \eqref{SystOndeFaibleAbstrait} we have for some constant $C>0$ 
 $$ \forall t>0,\quad \vert u(t)\vert^2+\vert v(t)\vert^2+\Vert u'(t)\Vert_*^2+\Vert v'(t)\Vert_*^2 \le C \frac{E_0} {t} $$ 
 with $$ E_0 = \Vert u(0)\Vert^2+\Vert v(0)\Vert^2+\vert u'(0)\vert^2+\vert v'(0) \vert^2.$$
\end{corollary}


\begin{proof}  We prove the theorem and its corollary together. The quadratic form

$$E_{-1}(u,v,,w,z)=\frac12 (\vert u\vert^2+\vert v\vert^2+\Vert w\Vert_*^2+\Vert z\Vert_*^2)+c\langle u,v\rangle_*$$ is equivalent to $$K(u,v,,w,z)=\vert u\vert^2+\vert v\vert^2+\Vert w\Vert_*^2+\Vert z\Vert_*^2 $$  and non-increasing along trajectories. In fact we have
\begin{equation}\label{InegKE-1} \frac{\lambda_1-c}{2\lambda_1} K(u,v,,w,z)\leq E_{-1}(u,v,,w,z)\leq \frac{\lambda_1+c}{2\lambda_1}K(u,v,,w,z)  \end{equation}
and
$$\frac{d}{dt}E_{-1}(u,v,u',v')=-\Vert   u'\Vert^2_*.$$
On the other hand, we have, introducing ${\cal{H}}(t): = H_\varepsilon(u(t),v(t),u'(t),v'(t))$
$${\cal{H}}'(t) = : \frac{d}{dt}H_\varepsilon (u,v,u',v')=-(1-p\varepsilon)\vert   u'\vert^2-  \varepsilon\lambda_1\Vert   v'\Vert_*^2-\varepsilon\frac{p-1}{2}\Vert   u\Vert^2-\varepsilon\frac{p-1}{2}\lambda_1\vert v\vert^2$$
$$+p\varepsilon c \langle u,v\rangle-\frac{(p+1)\lambda_1\varepsilon}{2c}\langle u', v\rangle -\varepsilon\lambda_1 c\langle v,u\rangle_*$$

$$\leq -(1-p\varepsilon)\vert  u'\vert^2-  \varepsilon\lambda_1\Vert   v'\Vert_*^2-\varepsilon\frac{p-1}{2}\Vert   u\Vert^2-\varepsilon\frac{p-1}{2}\lambda_1\vert v\vert^2+p\varepsilon \vert c\vert \vert u\vert \vert v\vert $$ $$+\frac{(p+1)\lambda_1\varepsilon}{2c}\vert u'\vert\vert v\vert+\varepsilon\lambda_1 \vert c\vert\Vert v \Vert_*\Vert u\Vert_*$$
$$\leq -(1-p\varepsilon)\vert  u'\vert^2-  \varepsilon\lambda_1\Vert  v'\Vert_*^2-\varepsilon\frac{p-1}{2}\Vert  u\Vert^2-\varepsilon\frac{p-1}{2}\lambda_1\vert v\vert^2$$ $$+\frac{(p+1) \vert c\vert }{\sqrt\lambda_1} \varepsilon\Vert   u\Vert \vert v\vert +\frac{(p+1)\lambda_1}{2\vert c\vert}\varepsilon\vert   u'\vert\vert v\vert$$

$$\leq -(1-p\varepsilon)\vert  u'\vert^2-  \varepsilon\lambda_1\Vert  v'\Vert_*^2-\varepsilon\frac{p-1}{2}\Vert   u\Vert^2-\varepsilon\frac{p-1}{2}\lambda_1\vert v\vert^2$$ $$+\frac{(p+1) \vert c\vert}{2\sqrt\lambda_1} \varepsilon\left(\frac{1}{\alpha}\Vert   u\Vert^2+\alpha \vert v\vert^2\right) +\frac{(p+1)\lambda_1}{2\vert c\vert}\varepsilon\vert  u'\vert\vert v\vert $$
(By Young's inequality, with $\alpha>0$ to be choosen later)
$${\cal{H}}'(t) \leq -(1-p\varepsilon)\vert   u'\vert^2-  \varepsilon\lambda_1\Vert  v'\Vert_*^2+\frac{(p+1)\lambda_1}{2\vert c\vert}\varepsilon\vert  u'\vert\vert v\vert$$
$$-\varepsilon\underbrace{\left(\frac{p-1}{2}- \frac{(p+1) \vert c\vert}{2\alpha\sqrt\lambda_1}   \right)}_{\delta=}\Vert   u\Vert^2-\varepsilon\underbrace{\left(\frac{p-1}{2}\lambda_1-\frac{(p+1) \vert c\vert\alpha}{2\sqrt\lambda_1}\right)}_{\zeta=}\vert v\vert^2$$
We choose $\alpha$ such that $\delta$ and $\zeta$ be positive, which is equivalent to $$\displaystyle \alpha\in (\frac{p+1}{p-1}\frac{\vert c \vert}{\lambda_1^{\frac12}}, \frac{p-1}{p+1}\frac{\lambda_1^{\frac32}}{\vert c\vert})$$ (this is made possible by \eqref{Hypsurp}.) 
Now by using Young's inequality, we get
$$\frac{(p+1)\lambda_1}{2\vert c\vert}\varepsilon\vert u'\vert\vert v\vert\leq \frac{(p+1)\lambda_1}{4\vert c\vert}\varepsilon\left(\frac{\vert   u'\vert^2}{\eta}+\eta\Vert v\Vert^2\right),\quad \forall \eta>0.$$
We choose $\eta$ such that $\frac{(p+1)\lambda_1 \eta}{4\vert c\vert}=\frac{\zeta}{2}.$ Therefore
$${\cal{H}}'(t)\leq -(1-\varepsilon(p+\frac{(p+1)\lambda_1}{4 \eta\vert c\vert}))\vert   u'\vert^2-  \varepsilon\lambda_1\Vert  v'\Vert_*^2 -\varepsilon \delta\Vert u\Vert^2-\varepsilon \frac{\zeta}{2}\vert v\vert^2.$$
Now we choose $\varepsilon$ small enough such that $1-\varepsilon(p+\frac{(p+1)\lambda_1}{4 \eta\vert c\vert})>0$, and finally we find  a constant $\gamma = \gamma (p, \varepsilon) >0$ such that for all $t\ge 0$
\begin{equation}\label{InegH'K} \frac{d}{dt}H_\varepsilon(u,v,u',v')= {\cal{H}}'(t) \leq -\gamma K(u,v,u',v')  \end{equation} At this stage the theorem is proved. We now deduce the corollary.
From  \eqref{InegH'K}, assuming $\varepsilon $ possibly smaller in order to achieve positivity of the quadratic form $H$, we get
$$\int_0^tK(u(s),v(s),u'(s),v'(s))\,ds\leq\frac{1}{\gamma}H_\varepsilon(u(0),v(0),u'(0),v'(0)).$$
Using  inequality \eqref{InegKE-1}, we obtain
$$\frac{2\lambda_1}{\lambda_1+c} \int_0^t  E_{-1}(u(s),v(s),u'(s),v'(s))  \,ds\leq \frac{1}{\gamma}H_\varepsilon(u(0),v(0),u'(0),v'(0)).$$
Now since $E_{-1}$ is nonincreasing, it follows 
$$E_{-1}(u(t),v(t),u'(t),v'(t)) \leq  \frac{\lambda_1+c}{2\lambda_1\gamma }\frac{1}{t}H_\varepsilon(u(0),v(0),u'(0),v'(0)).$$
Using once again inequality \eqref{InegKE-1} we get
$$K(u(t),v(t),u'(t),v'(t)) \leq \frac{\lambda_1+c}{(\lambda_1-c)\gamma }\frac{1}{t}H_\varepsilon(u(0),v(0),u'(0),v'(0))\le C \frac{E_0} {t}.$$

\end{proof}

 \begin{remark} {\rm We recover here one of the main results of \cite{ACK02} by a Liapunov function approach. It seems that many indirect stabilization results can be proved by the same method. All the results involving different usual norms on both sides of the inequality can be deduced from the corollary by using $A-$ invariance, induction or interpolation. The theory will be complete as soon as optimality of the negative power of $t$ is established, and the comparison with similar simpler problems makes it look  reasonable. } \end{remark} 

\bibliographystyle{amsplain}

\providecommand{\bysame}{\leavevmode\hbox to3em{\hrulefill}\thinspace}

\end{document}